\documentclass[12pt]{amsart}

\usepackage{amsmath,amssymb,amsthm}
\usepackage{amsfonts}
\usepackage[mathscr]{eucal}
\usepackage[dvips]{graphics,color}

\usepackage{amsmath,amssymb,amsthm}
\usepackage{amsfonts}
\usepackage[mathscr]{eucal}
\newcommand{\rd}{{\mathbb R^d}}

\newcommand{\rr}{{\mathbb R}}

\def\E{{\mathbb E}}
\newtheorem{thm}{Theorem}
\numberwithin{thm}{section}
\newtheorem{lemma}[thm]{Lemma}
\newtheorem{remark}{Remark}

\newtheorem{theorem}{Theorem}[section]

\numberwithin{equation}{section}
\numberwithin{remark}{section}
\numberwithin{proposition}{section}
\newtheorem{corollary}{Corollary}[section]

\newcommand{\RR}[1]{\mathbb{#1}}

\def\N{{\mathbb N}}

\def\E{{\mathbb E}}

\begin{document}

\thispagestyle{empty}

\title{\bf Stochastic solutions of a class of Higher order Cauchy problems in $\rd$}
\author{Erkan Nane}
\address{Erkan Nane, Department of Mathematics and Statistics, 221 Parker Hall,
Auburn University, Auburn, AL 36849}
\email{nane@auburn.edu}

\begin{abstract}
 We study solutions of a class of  higher order partial differential equations  in bounded domains. These partial differential equations appeared first time in the papers of  Allouba and Zheng \cite{allouba1},  Baeumer, Meerschaert and Nane \cite{bmn-07}, Meerschaert, Nane and Vellaisamy \cite{MNV}, and Nane \cite{nane-h}. We express the solutions by subordinating a killed Markov process by a hitting time of a stable subordinator of index $0<\beta <1$, or by  the absolute value of a symmetric $\alpha$-stable process with $0<\alpha\leq 2$, independent of the Markov process. In some special cases we represent the solutions by running composition of $k$ independent Brownian  motions, called $k$-iterated Brownian motion for an integer $k\geq 2$. We make use of a connection between fractional-time diffusions and higher order partial differential equations established first by Allouba and Zheng \cite{allouba1} and later extended in several directions by Baeumer, Meerschaert and Nane \cite{bmn-07}.

\end{abstract}



\maketitle

\section{Introduction and statement of main results}

In recent years, there have been two lines of study of the stochastic solutions of partial differential equations (PDE's): higher order Cauchy problems \cite{allouba2,allouba3, allouba1,bmn-07,MNV,nane-h} and time fractional Cauchy problems \cite{fracCauchy, gammagt1,Zsolution,MNV}. We will use the equivalence of these two types of Cauchy problems on $\rd$ and on bounded domains with Dirichlet boundary conditions to get classical as well as stochastic  solutions of a class of higher order Cauchy problems that appeared in \cite{bmn-07,MNV,nane-h}.

In this paper we suppose that Brownian motion  has variance $2t$ (or  that    time clock is  twice the speed of a standard Brownian motion). We express the solutions of these Cauchy problems  using $k$-iterated Brownian motions for  an integer $k\geq 2$. There are two ways to define $k$-iterated Brownian motions. The first one is just to let
\begin{equation}\label{iterated-bm1}
I_{k}(t)=B_1(|B_2(|B_3(|\cdots (|B_{k}(t)|)\cdots|)|)|)
\end{equation}
where $B_j$'s are independent real-valued Brownian motions all started at $0$. In $\rd$, one takes $B_1$ to be an $\rd$-valued Brownian motion with independent components. In this case we denote $k$-iterated Brownian motion by $I^d_{k}(t)$.

To define the second version of $k$-iterated Brownian motions,
let $X^{+}(t)$, $X^{-}(t)$
 be  independent
one-dimensional Brownian motions, all started at $0$.
 Two-sided Brownian motion is defined to be
\[ X(t)=\left\{ \begin{array}{ll}
X^{+}(t), &t\geq 0\\
X^{-}(-t), &t<0.
\end{array}
\right. \]
Then the second version of $k$-iterated Brownian motion is defined as

\begin{equation}\label{iterated-bm2}
 J_{k}(t)=X_1(X_2(X_3(\cdots (X_{k}(t))\cdots)))
\end{equation}
where $X_j$'s are independent real-valued two-sided Brownian motions all started at $0$. In $\rd$, one takes $X_1$ to be an $\rd$-valued two-sided Brownian motion with independent components. In this case we denote $k$-iterated Brownian motion by $J^d_{k}(t)$

For $k=2$, both of these processes $I_2$  and $J_2$ were called iterated Brownian motion (IBM) and they have been studied by several researchers; see, for example \cite{allouba2,allouba1,burdzy1,burdzy2, koslew,nane,nane2,nane3, nourdin} and references therein.

Recently, Aurzada and Lifshits \cite{aurzada} studied the small ball probability for $I_k, J_k$. Arcones \cite{arcones} studied the large deviation for $k$-iterated Brownian motions.

The classical well-known connection of a PDE and a stochastic process
is the Brownian motion and heat equation connection. Let $X({t})\in
\rd$ be Brownian motion started at $x$. Then the function
$$
u(t,x)=\E_{x}[f(X({t}))]
$$
is the unique solution of  the  Cauchy problem under mild conditions on $f$
\begin{equation*}\begin{split}
\frac{\partial}{\partial t}u(t,x)\ =  \Delta u(t,x) ;\ &\quad
u(0,x) = f(x)
\end{split}\end{equation*}
for $t>0$ and $x\in \rd$.
There is a similar connection for any Markov process where we replace $\Delta$  with the generator of the Markov process; see, for example, \cite{ABHN,bass, davies, sato}.

 Allouba and Zheng \cite{allouba1} and DeBlassie  \cite{deblassie} obtained the PDE connection of 2-iterated Brownian motion.
They showed that for $Z(t)=I_2(t),\ \mathrm{or}\ J_2(t)$
 $$
 u(t,x)=\E_{x}[f(Z({t}))]
 $$
 solves the Cauchy problem
\begin{equation}\begin{split}\label{ibm-pde0}
\frac{\partial}{\partial t}u(t,x)\ =
\frac{{\Delta}f(x)}{\sqrt{\pi t}}\ + \
{\Delta}^{2}u(t,x);& \quad
u(0,x)= f(x)
\end{split}\end{equation}
for $t>0$ and $x\in \rd$.
The non-Markovian property of IBM is reflected by the appearance
of the initial function $f(x)$ in the PDE.
The methods of Allouba and Zheng are more general and they  allow one to replace $\Delta$ with the generator of a Markov process.
Another important characteristic of the Allouba-Zheng work is  its
setup. 
Let $X^x$ be a continuous Markov process started at $x$ and let $B(t)$ be a Brownian motion independent of $X$. They call $X^x(|B(t)|)$ a Brownian time process(BTP). Further more: 
\begin{enumerate}
\item  They introduced-for the first time-a large class of stochastic processes including kEBTPs  that is obtained by taking  at random one of the $k$ copies of independent Markov process $X^x$ on each excursion interval of $|B(t)|$. By this method some  important  classes of processes are obtained  including iterated Brownian motion of Burdzy \cite{burdzy1} in the case $k=2$,  BTPs in the case of $k=1$, Markov snake of Le Gall \cite{legal1,legal1,legal3}, and  a process that is intermediate between iterated Brownian motion of Burdzy and Markov snake of Le Gall in the case of taking limit $k\to\infty$.
\item  Their formulation and approach handle and link to fourth order PDEs not
just IBMs, but a much larger class (in (1)) containing many other interesting
new processes.
\end{enumerate}


In the absence of the initial function in the PDE \eqref{ibm-pde0} this problem was studied in  \cite{funaki, hoc-ors,ben-roy-val}.

Nigmatullin \cite{nigmatullin} gave a Physical derivation of fractional
diffusion
\begin{equation}\label{frac-derivative-0}
\frac{\partial^\beta}{\partial t^\beta}u(t,x) = {L_x}u(t,x); \quad
u(0,x) = f(x)
\end{equation}
for $t>0$ and $x\in \rd$, where $0< \beta <1$ and $L_x$ is the
generator of some continuous Markov process $X_0(t)$ started at
$x=0$. Here $\partial^{\beta} g(t)/\partial t^\beta $ is the Caputo
fractional derivative in time, which can be defined as the inverse
Laplace transform of $s^{\beta}\tilde{g}(s)-s^{\beta -1}g(0)$, with
$\tilde{g}(s)=\int_{0}^{\infty}e^{-st}g(t)dt$ the usual Laplace
transform. Zaslavsky \cite{zaslavsky} used \eqref{frac-derivative-0} to model Hamiltonian chaos.

Baeumer and Meerschaert \cite{fracCauchy} and Meerschaert and Scheffler \cite{limitCTRW}
show that, in the case $p(t,x)=T(t)f(x)$ is a bounded continuous
semigroup on a Banach space (with corresponding process $X(t)$, $E^\beta(t)=\inf \{u:\ D(u)>t\}$, $D(t)$ is a stable subordinator with index $0<\beta<1$), the formula
\begin{eqnarray}
u(t,x)& =&\E_x(f(X({E^\beta(t)})))=\frac{t}{\beta}\int_{0}^{\infty}p(s,x)g_{\beta}(\frac{t}{s^{1/\beta}})s^{-1/\beta
-1}ds \nonumber 
\end{eqnarray}
yields the unique solution to the  fractional Cauchy problem \eqref{frac-derivative-0}.
  Here $g_\beta(t)$ is the smooth
density of the stable subordinator $D(1)$, such that the Laplace
transform $\tilde g_\beta(s)=\int_0^\infty
e^{-st}g_\beta(t)\,dt=e^{-s^\beta}$.

Allouba and Zheng \cite{allouba1} were also
the first to establish a connection between their class of BTPs and the time half-derivative through their BTP half-derivative generator (see their Theorem 0.5), in
addition to connecting their BTPs to 4th order PDEs. Essentially, Allouba and Zheng \cite{allouba1} show that BTPs are  also stochastic solution of   \eqref{frac-derivative-0} in the case $\beta=1/2$ and $L_x$ is a second order elliptic differential operator of divergence form.

Later,
 Orsingher and Beghin \cite{OB,OB1} show that
$u(t,x)=\E_x(f(I_{k+1}(t)))$ is the solution of
\begin{equation}\label{frac-derivative-2n}
\frac{\partial^{1/2^k}}{\partial t^{1/2^k}}u(t,x) =
\frac{\partial^2}{\partial x^2}u(t,x); \quad u(0,x) = f(x)
\end{equation}
for $t>0$ and $x\in \rr$.

We will denote the
Laplace, Fourier, and Fourier-Laplace transforms (respectively) by:
\begin{equation*}\begin{split}
\tilde{u}(s,x)&=\int_{0}^{\infty}e^{-st}u(t, x)dt ; \\
\hat{u}(t,k)&=\int_{\RR{R}^{d}}e^{-ik\cdot x}u(t, x)dx ; \\
\bar{u}(s,k)&=\int_{\RR{R}^{d}}e^{-ik\cdot x}\int_{0}^{\infty}e^{-st}
u(t, x)dtdx .
\end{split}\end{equation*}

Using Fourier-Laplace transform method,
Baeumer, Meerschaert, and Nane \cite{bmn-07} showed the equivalence of a class of Higher order Cauchy problems and time  fractional Cauchy problems:   Suppose that $X(t)=x+X_0(t)$ where $X_0(t)$ is a L\'evy process in $\rd$
starting at zero.  If $L_x$ is the generator
of the semigroup $p(t,x)=\E_x[(f(X(t)))]$, then for any $f\in D(L_{x})$, the domain of $L_x$, and for any $m=2,3,4,\ldots$, both the Cauchy problem
\begin{equation}\begin{split}\label{one-third-m}
\frac{\partial u(t,x)}{\partial t} & =\sum_{j=1}^{m-1} \frac{t^{j/m-1}}{\Gamma(j/m)} L_x^j f(x)
+ L_x^m u(t,x);\quad   u(0,x) =  f(x)
\end{split}\end{equation}
and the fractional Cauchy problem
\begin{equation}\label{frac-derivative-m}
\frac{\partial^{1/m}}{\partial t^{1/m}}u(t,x) =
{L_x}u(t,x); \quad u(0,x) = f(x),
\end{equation}
have the same unique solution given by
$$
u(t,x) =\E_x (f(X(E^{1/m}(t))))=\int_0^\infty p((t/s)^{1/m},x) g_{1/m}(s)\,ds.
$$
Considering \eqref{frac-derivative-2n}, and the equivalence of \eqref{one-third-m} and \eqref{frac-derivative-m}, it is natural to expect the following Theorem. It establishes the PDE connection of $k-$iterated Brownian motion which extends PDE connection of $2$-iterated Brownian motion (IBM) due to Allouba and Zheng \cite{allouba1} and DeBlassie \cite{deblassie}.

\begin{theorem}\label{pde-conn-n}
Suppose that $X(t)=x+X_0(t) \in \rd$  where $X_0(t)$ is a L\'evy process
starting at zero.  If $L_x$ is the generator
 of the
semigroup $T(t)f(x)=\E_x[(f(X(t)))]$, then for any
 $f\in D(L_{x})$, domain of $L_x$, $u(t,x)=\E_x(f(X(|I_{k}(t)|)))$ is the
 unique solution of the Cauchy problems \eqref{one-third-m} and \eqref{frac-derivative-m} with $m=2^k$. If $Z(t)$ is a two-sided L\'evy process with independent copies of $X$ for positive and negative times then $u(t,x)=\E_x(f(Z(J_{k}(t))))$ is also the
 unique solution of the Cauchy problems \eqref{one-third-m} and \eqref{frac-derivative-m} with $m=2^k$.
\end{theorem}

\begin{remark}
Only trivial extensions to special cases of the formulation in Allouba and Zheng \cite{allouba1} lead to the fact that the
expressions with $I_k(t)$ and $J_k(t)$ in Theorem \ref{pde-conn-n} are interchangeable in the solution expressions of
the given PDEs ( see  Theorem 0.1 and its proof  in Allouba and 
Zheng \cite{allouba1}).
\end{remark}

Let $D $ be a  domain in $\RR{R}^d$.
 We define the following
spaces of functions.
\begin{eqnarray}
C(D)&=&\{u:D\to \rr :\ \ u\  \mathrm{is\
continuous}\};\nonumber\\
C(\bar D)&=&\{u:\bar D\to \rr :\ \ u \ \mathrm{is \  uniformly\
continuous}\};\nonumber\\
C^j(D)&=&\{u:D\to \rr :\ \ u\  \mathrm{is}\ j\mathrm{-times\  continuously\ differentiable}\};\nonumber\\
C^j(\bar D)&=&\{u\in C^j(D) :\ \ D^\gamma u\  \mathrm{is\
uniformly  \  continuous \ for \ all}\  |\gamma|\leq j\}.\nonumber
\end{eqnarray}
Thus, if $u\in C^j(\bar D)$, then $D^\gamma u$ continuously
extends to $\bar D $ for each multi-index $\gamma$\  with
$|\gamma|\leq j$.

We define the spaces of functions $C^\infty(D)=\cap_{j=1}^\infty
C^j(D)$ and $C^\infty(\bar D)=\cap_{j=1}^\infty C^j(\bar D)$.

Also, let $ C^{j,\alpha}( D)$ ($ C^{j,\alpha}(\bar D)$) be
the subspace of $ C^{j}( D)$ ($ C^{j}(\bar D)$) that consists of
functions whose $j$-th order partial derivatives are uniformly
H\"older continuous with exponent $\alpha$ in $D$.
For simplicity, we will write
$$
C^{0,\alpha}(D)=C^\alpha(D), \ \ \ C^{0,\alpha}(\bar D)=C^\alpha(\bar D)
$$
with the understanding that $0<\alpha<1$ whenever this notation is used, unless otherwise stated.

 We
use $C_c(D), C_c^j(D),C^{j,\alpha}_c( D) $ to denote those
functions in $C(D), C^j(D), C^{j,\alpha}( D)$ with compact
support.

A subset $D$ of $\rd$ is an $l$-dimensional manifold with boundary
if every point of $D$ possesses a neighborhood diffeomorphic to an
open set in the space $H^l$, which is the upper half space in
$\RR{R}^l$. Such a diffeomorphism is called a local
parametrization of $D$. The boundary of $D$, denoted by $\partial
D $, consists of those points that belong to the image of the
boundary of $H^l$ under some local parametrization. If the
diffeomorphism and its inverse are $C^{j,\alpha}$ functions, then we
write $\partial D \in C^{j,\alpha}$.

Since we are working on a bounded domain, the Fourier transform methods in \cite{Zsolution} are not useful.  Instead we will employ Hilbert space methods used in \cite{MNV}.  Hence, given a complete orthonormal basis $\{\psi_n(x)\}$ on $L^2(D)$, we will call
\begin{equation*}\begin{split}
\bar{u}(t,n)&=\int_{D}\psi_n(x)u(t,x)dx ; \\
\hat{u}(s,n)&=\int_{D}\psi_n(x)\int_{0}^{\infty}e^{-st} u(t,
x)dtdx = \int_{D} \psi_{n}(x) \tilde{u}(s, x) dx.
\end{split}\end{equation*}
the $\psi_n$, and $\psi_n$-Laplace transforms, respectively.  Since $\{\psi_n\}$ is a complete orthonormal basis for $L^2(D)$, we can invert the $\psi_n$-transform
\[u(t,x)=\sum_n \bar{u}(t,n) \psi_n(x)\]
for any $t>0$, where the sum converges in the $L^2$ sense (e.g., see \cite[Proposition 10.8.27]{Royden}).

 Mittag-Leffler function is defined by
$E_\beta(z)=\sum_{k=0}^{\infty}\frac{z^k}{\Gamma (1+\beta k)} ,$ $0<\beta<1$;
see, for example, \cite{MG-ML}.

Let $\beta\in (0,1)$, $D_\infty=(0,\infty )\times D$  and define

\begin{eqnarray}
\mathcal{H}_\Delta(D_\infty)&\equiv & \left\{u:D_\infty\to \rr :\ \
\frac{\partial}{\partial t}u, \frac{\partial^\beta}{\partial t^\beta}u, \Delta u\in C(D_\infty),\right.\nonumber\\
& &\left. \left|\frac{\partial}{\partial t}u(t,x)\right|\leq g(x)t^{\beta -1}, g\in L^\infty(D), \ t>0 \right\}.\nonumber
\end{eqnarray}

Let $\tau_D(X)=\inf\{t\geq 0: \ X(t)\notin D\}$ be the first exit time of the process $X$
 from $D$. We denote by $\{\phi_n, \lambda_n, n\geq 1\}$ the set of eigenvalues $\lambda_n$ and corresponding eigenfunctions $\phi_n$ of the
Laplacian $\Delta$ with Dirichlet boundary conditions:
$$\Delta \phi_n=-\lambda_n\phi_n \ \mathrm{ in }\ D;\ \phi|_{\partial D}=0.$$

We will write $u\in C^k(\bar D)$ to mean that for each fixed $t>0$, $u(t,\cdot)\in C^k(\bar D)$, and
 $u\in C_b^k(\bar D_\infty)$ to mean that $u\in C^k(\bar D_\infty)$ and is bounded.

 Extending the Fourier-Laplace transform method to bounded domains, Meerschaert, Nane
 and Vellaisamy \cite{MNV} gave a stochastic as well as  an analytic solution to fractional
 Cauchy problem \eqref{frac-derivative-0} in bounded domains:
 Let $0<\gamma<1$. Let $D$ be a bounded domain with $\partial D \in
 C^{1,\gamma}$, and
  $T_D(t)$ be the killed semigroup of Brownian motion  $\{X_t\}$ in $D$. Let $E^\beta(t)$ be the process inverse
  to a stable subordinator of index $\beta\in (0,1)$ independent of $\{X(t)\}$.
  Let $f\in D(\Delta)\cap C^1(\bar D)\cap C^2(D)$ for which the
  eigenfunction expansion (of $\Delta f$) with respect to the complete orthonormal basis $\{\phi_n:\ n\in \N \}$ converges uniformly and absolutely.
Then the unique (classical) solution of


\begin{eqnarray}
 u & \in &
\mathcal{H}_\Delta(D_\infty)\cap C_b(\bar D_\infty) \cap C^1(\bar D); \nonumber\\
\frac{\partial^\beta}{\partial t^\beta}u(t,x) &=&
\Delta u(t,x);  \  \ x\in D, \ t>0;\label{frac-derivative-bounded-d}\\
u(t,x)&=&0, \ x\in \partial D, \ t>0; \nonumber\\
u(0,x)& =& f(x), \ x\in D\nonumber
\end{eqnarray}
is given by
\begin{eqnarray}
u(t,x)&=&\E_{x}[f(X(E^\beta({t})))I( \tau_D(X)> E^\beta(t))]\nonumber\\
&=& \frac{t}{\beta}\int_{0}^{\infty}T_D(l)f(x)g_{\beta}
(tl^{-1/\beta })l^{-1/\beta -1}dl= \int_{0}^{\infty}T_D((t/l)^\beta)f(x)g_{\beta}
(l)dl \nonumber\\
&=&
\sum_{n=1}^\infty \bar f(n)\phi_n(x)E_\beta(-\lambda_n
t^\beta).\label{mittag-series-0}
\end{eqnarray}
\begin{remark}
The analytic representation \eqref{mittag-series-0} of the solution is due to Agrawal \cite{agrawal} in the case $D=(0, M)$ is an interval in $\rr$.
For $\beta=1$, the study of the Cauchy problem \eqref{frac-derivative-bounded-d}  boils down to studying heat equation in  bounded domains with Dirichlet boundary conditions which has solution \eqref{mittag-series-0}. In this case $E_1(-\lambda_n
t)=e^{-t\lambda_n}$. This is valid under much less requirements on the initial function and the regularity of the boundary, see for example \cite{bass}. The solution to \eqref{frac-derivative-bounded-d} is also  given by \eqref{mittag-series-0} if we replace Brownian motion with a diffusion process in which case the Laplacian $\Delta$ should be replaced with the diffusion operator, see Meerschaert, et. al. \cite{MNV}.
\end{remark}

Let
\begin{eqnarray}
\mathcal{H}_{\Delta^m}(D_\infty)&\equiv & \bigg\{u:D_\infty\to \rr :\ \big|\frac{\partial}{\partial t}u(t,x)\big|\leq g(x)t^{1/m-1}, g\in L^\infty(D), \ t>0,\nonumber\\
&&\frac{\partial}{\partial t}u,  \Delta ^k u\in C(D_\infty), k=1,\cdots, m, \Delta^{k} u \in C^1(\bar D), k=1,\cdots, m-1  \bigg\}.\nonumber
\end{eqnarray}


Using the equivalence of fractional Cauchy problem \eqref{frac-derivative-bounded-d} with the  Higher order Cauchy problems \eqref{one-third-m}  with  the correct Dirichlet type Boundary conditions we obtain the second main result in this paper.
 \begin{theorem}\label{ibm-pde-thm}
Let $m=2,3,\cdots,$ be an integer.
Let $D$ be a bounded domain with $\partial D \in
 C^{1,\gamma}$, and
  $T_D(t)$ be the killed semigroup of Brownian motion  $\{X_t\}$ in $D$.
Let $\{E^{1/m}(t)\}$ be the process inverse  to
a stable subordinator of index $1/m$ independent of $\{X(t)\}$.
 Let $f\in D(\Delta)\cap C^{2m-3}(\bar D)\cap C^{2m-2}(D) (\subset L^2(D))$ be  such that the eigenfunction
 expansion  of $\Delta^{m-1} f$ with respect to $\{\phi_n: \ n\geq 1\}$
 converges absolutely and uniformly.
Then  the (classical)  solution of


\begin{eqnarray}
u&\in & \mathcal{H}_{\Delta ^m}(D_\infty)\cap C_b(\bar D_\infty)\cap C^1(\bar D);\nonumber\\
\frac{\partial u(t,x)}{\partial t} & =&\sum_{j=1}^{m-1} \frac{t^{j/m-1}}{\Gamma(j/m)} \Delta^j f(x)
+ \Delta^m u(t,x), \  \ x\in D, \ t> 0;\label{ibm-pde-bounded-02}\\
u(t,x)&=&{\Delta}^lu(t,x)=0, \ t\geq 0, \ x\in \partial D, \ l=1,\cdots m-1; \nonumber\\
u(0,x) &=& f(x),\ x\in D\nonumber
\end{eqnarray}
is given by
\begin{eqnarray}
u(t,x)&=&\E_{x}[f(X(E^{1/m}({t})))I( \tau_D(X)> E^{1/m}(t))]\nonumber\\
&=&\E_{x}[f(X(E^{1/m}({t})))I( \tau_D(X(E^{1/m}))>
t)]\label{higher-order-solution}\\
&=& tm\int_{0}^{\infty}T_D(l)f(x)g_{1/m}
(tl^{-m})l^{-m -1}dl\nonumber\\
&=&
\sum_{n=1}^\infty \bar f(n)\phi_n(x)E_{1/m}(-\lambda_n
t^{1/m}).\nonumber
\end{eqnarray}
In the case $m=2^k$ for some integer $k\geq 1$, the solution to \eqref{ibm-pde-bounded-02} is also given by $u(t,x)=\E_{x}[f(X(|I_k(t)|))I( \tau_D(X)> |I_k(t)|)].$ For   $\beta=1/2^k$ the solution to \eqref{frac-derivative-bounded-d} is also given by $u(t,x)=\E_{x}[f(X(|I_k(t)|))I( \tau_D(X)> |I_k(t)|)].$
 \end{theorem}
 Meerschaert, Nane and Vellaisamy \cite{MNV} proved this theorem in the case $m=2$.
\begin{remark}
Theorem  \ref{ibm-pde-thm} also holds with the version of $k$-iterated  Brownian motion $J_k(t)$. Here, the outer process $X(t)$ is a two-sided
Brownian motion and $J_k(t)$ is an
independent $k$-iterated  Brownian motion. In this case,  using a simple conditioning argument,
we can show that
the function
$$u(t,x)=E_x[f(X(J_k(t)))I(  -\tau_D(X^-)<J_k(t)<\tau_D(X^+))]$$
reduces to Equation \eqref{higher-order-solution} and hence is also a solution to both Cauchy problems
\eqref{ibm-pde-bounded-02} and \eqref{frac-derivative-bounded-d} with $m=2^k$.
\end{remark}

In \cite{nane-h}, we studied the Cauchy problems that can be solved by running $\alpha$-time processes with $0<\alpha\leq 2$. An $\alpha$-time process is a Markov process in which the time parameter is replaced with the
 absolute value of an independent symmetric $\alpha$-stable process $Y$ with $0<\alpha\leq 2$.

  As a special case in  Nane \cite[Theorem 2.1]{nane-h} we established:
Let
 $\{X(t)\}$ be a continuous Markov process with generator $\mathcal{A}$, and let $\{Y(t)\}$ be a Cauchy process independent of $\{X(t)\}$. Let $f$ be a bounded measurable
function in the domain of $\mathcal{A}$, with $D_{ij}f$ bounded and
H\"{o}lder continuous for all $1\leq i,\ j \leq d$. Then
$u(t,x)=\E_x[f(X(|Y(t)|)]$ is a solution of
\begin{eqnarray}
\frac{\partial^{2}}{\partial t^{2}}u(t,x)\ & = &-\frac{2\mathcal{A}
f(x)}{\pi t}\ - \ \ \mathcal{A} ^{2} u(t,x);\quad
 u(0,x)  = f(x)\nonumber
\end{eqnarray}
for $t>0, \ \ x\in \RR{R}^{d}$.

This reduces to {\bf nonhomogeneous wave equation} in the case $X$ is another Cauchy process independent of $Y$, the generator $\mathcal{A}=-(-\Delta)^{1/2}$, fractional Laplacian, i.e.,
$u(t,x)=\E_x[f(X(|Y(t)|)]$ is a solution of
\begin{eqnarray}
\frac{\partial^{2}}{\partial t^{2}}u(t,x)\ & = &\frac{2(-\Delta)^{1/2}
f(x)}{\pi t}\ + \ \ \Delta u(t,x);\quad
 u(0,x)  = f(x)\nonumber
\end{eqnarray}
for $t>0, \ \ x\in \RR{R}^{d}$.
This is one of the most interesting PDE connections of these iterated processes.

Let $$
\mathcal{K}_{\Delta ^2}=\{u:D_\infty\to \rr:\ \ \frac{\partial^2}{\partial t^2}u, \Delta u, \Delta^2 u\in C(D_\infty) \}
$$
 In bounded domains we obtain the following
  \begin{theorem}\label{cauchy-bounded}
 Let $D$ be a bounded domain with $\partial D\in C^{1,\gamma},  0< \gamma <1.$
Let
 $\{X(t)\}$ be a Brownian motion in $\rd$ with independent components, and let $\{Y(t)\}$ be a Cauchy process independent of $\{X(t)\}$. Let $f\in D(\Delta)\cap C^2(D)$ for which the eigenfunction expansion of $\Delta f$  with respect to the complete orthonormal basis $\{\phi_n:\ n\geq 1\}$ converges absolutely and uniformly.
Then
$u(t,x)=\E_x[f(X(|Y(t)|)I(\tau_D(X)>|Y(t)|)]$ is a solution of
\begin{eqnarray}
u &\in &  \mathcal{K}_{\Delta ^2}\cap C_b(D_\infty)\cap C^1(\bar D)\nonumber\\
\frac{\partial ^{2}}{\partial t^{2}}u(t,x)\ & =  & -\frac{2\Delta f(x)}{\pi t}\ - \ \Delta^{2}u(t,x),
\  \  \ t>0, \ \ x\in D\label{PDE-CONNECT2-bounded}\\
u(0,x) & =  & \ f(x), \ \ \ \ \ x \in D,\nonumber\\
u(t,x)& = &\Delta u(t,x) = 0,\ \ x\in \partial D, \ t\geq 0.\nonumber
\end{eqnarray}

\end{theorem}

\begin{remark}
In Theorem \ref{ibm-pde-thm}, the solution is expressed by subordinating a killed process by  an increasing process $E^{1/m}(t)$, inverse process to a stable subordinator of index $1/m$. And we see from Theorem \ref{ibm-pde-thm} that killing the outer process and then subordinating, or subordinating and then killing gives the same solution of the Cauchy problem \eqref{ibm-pde-bounded-02}. The solution in the case $m=2^k$ is also given by subordinating the killed process by $k$-iterated Brownian motion. Is there an increasing  process $A(t)$ such that we get a similar relation for the solution of \eqref{PDE-CONNECT2-bounded}?
\end{remark}

\begin{remark}
We discuss how to extend Theorems \ref{pde-conn-n}, \ref{ibm-pde-thm} and \ref{cauchy-bounded} to other continuous Markov processes  in section 6.
\end{remark}
This paper is organized as follows. In section 2 we give some preliminaries. Section 3 is devoted to the proof of Theorem \ref{pde-conn-n}. We prove Theorem \ref{ibm-pde-thm} in section 4. Theorem \ref{cauchy-bounded} is proved in section 5. We  also state and prove a theorem for Cauchy problems that can be solved by running a Brownian motion subordinated to the absolute value of a symmetric $\alpha$-stable process with $\alpha\in (0,2)$ a rational, and $\alpha\neq 1$. We discuss extensions of the theorems proved in this paper to other types of Markov process in section 6.
\section{preliminaries}

Let $X_0(t)$ be a L\'evy process started at zero and
$X(t)=x+X_0(t)$ for $x\in\rd$, the generator $L_x$ of the
semigroup $T(t)f(x)=\E_x[f(X(t))]$ is a pseudo-differential
operator \cite{applebaum,Jacob,schilling} that can be explicitly
computed by inverting the L\'evy representation. The L\'{e}vy
process $X_0(t)$ has characteristic function $$\E[\exp(ik\cdot
X_0(t))]=\exp(t\psi(k))$$ with
$$
\psi (k)=ik\cdot a-\frac{1}{2}k\cdot Qk+ \int_{y\neq 0}\left(
e^{ik\cdot y}-1-\frac{ik\cdot y}{1+||y||^{2}}\right)\nu(dy),
$$
where $a\in \RR{R}^{d}$, $Q$ is a nonnegative definite matrix, and
$\nu$ is a $\sigma$-finite Borel measure on $\RR{R}^{d}$ such
that
$$
\int_{y\neq 0}\min \{1,||y||^{2}\}\nu(dy)<\infty;
$$
see for example \cite[Theorem 3.1.11]{RVbook} and \cite[Theorem
1.2.14]{applebaum}. Let
$$\hat{f}(k)=\int_{\RR{R}^{d}}e^{-ik\cdot
x}f( x)\,dx$$
denote the Fourier transform. Theorem 3.1 in
\cite{fracCauchy} shows that $L_x f(x)$ is the inverse Fourier
transform of $\psi(k)\hat{f}(k)$ for all $f\in D(L_x)$, where
$$
D(L_x)=\{ f\in L^{1}(\RR{R}^{d}):\ \psi(k)\hat{f}(k)=\hat{h}(k)\
\exists \ h \in L^{1}(\RR{R}^{d})  \},
$$
and
\begin{equation}\begin{split}\label{pseudoDO}
L_x f(x)&= a\cdot\nabla f(x) +\frac{1}{2}\nabla \cdot Q\nabla
f(x)\\
&+ \int_{y \neq  0} \left(  f(x+y)-f(x)-\frac{\nabla
f(x)\cdot y}{1+y^{2}} \right)\nu(dy)
\end{split}\end{equation}
for all $f\in W^{2,1}(\RR{R}^{d})$, the Sobolev space of
$L^{1}$-functions whose first and  second partial derivatives are
all $L^{1}$-functions.  This includes the special case where
$X_0(t)$ is an operator L\'evy motion.  We can also write
$L_x=\psi(-i\nabla)$ where $\nabla=(\partial/\partial
x_1,\ldots,\partial/\partial x_d)'$.  For example, if $X_0(t)$ is
spherically symmetric stable then $\psi(k)=-D\|k\|^\alpha$ and
$L_x=-D(-\Delta)^{\alpha/2}$, a fractional derivative in space,
using the correspondence $k_{j}\to -i\partial/\partial x_j$ for
$1\leq j\leq d$.





We say that $D$ satisfies an {\em exterior cone condition} at a
fixed point $x_0\in \partial D$ if there exists a finite right
circular cone $V=V_{x_0}$ with vertex $x_0$ such that $\bar D\cap
V_{x_0}=x_0$, and
 a {\em uniform exterior cone condition}
if $D$ satisfies an exterior cone condition at every point $x_0\in
\partial D$  and the cones $V_{x_0}$ are all congruent to some
fixed cone $V$.

Let $\{X_t\}$ be a Brownian motion in $\rd$ and $A\subset \rd$. Let
$T_A= \inf \{t>0: \ X_t\in A \}$ be the first hitting time of the set
$A$. We say that a point $y$ is {\em regular} for a set $A$ if
$P_y[T_A=0]=1$.
Note that a point $y$ is regular for a set $A$ provided, starting at $y$,
the process does not go a positive length of time before hitting
$A$.

The right condition for  the existence of the solution to the
Dirichlet problem turns out to be that every point of $\partial D$
is regular for $D^C$ (cf. \cite[Section II.1]{bass}).

If a domain satisfies a uniform exterior cone condition, then every point of $\partial D$ is regular for $D^C$.



Let $\partial D\in C^1$. Then  at each point $x\in \partial D$
there exists a unique outward pointing unit vector
$$
\theta(x)=(\theta_1(x), \cdots, \theta_d(x)).
$$

Let $u\in C^1(\bar{D})$, the set of functions which have
continuous extension of the first derivative up to the boundary.
Let
$$
D_\theta u=\frac{\partial u}{\partial \theta}=\theta\cdot\nabla u,
$$
denote the directional derivative,  where $\nabla u$ is the
gradient vector of $u$.

 Now we recall Green's first and second identities (see, for example, \cite[Section 2.4]{gilbarg-trudinger}). Let $u,v\in C^{2}(D)\cap
 C^1(\bar{D})$. Then

$$
\int_D \frac{\partial u}{\partial x_i } vdx=\int _{\partial D}
uv\theta_ids-\int_D u\frac{\partial v}{\partial x_i }dx \ \
(\mathrm{integration\ by \ parts\ formula}),
$$
 $$
\int_D\nabla v \cdot\nabla u dx =-\int_D u\Delta v dx
+\int_{\partial D}\frac{\partial v}{\partial \theta}uds\ \
(\mathrm{Green's\ \ first\ \  identity}),
 $$
$$
\int_D [u\Delta v-v\Delta u]dx=\int_{\partial
D}\left[u\frac{\partial v}{\partial \theta}-v\frac{\partial
u}{\partial \theta}\right]ds,\ \ (\mathrm{Green's\ \ second\ \
identity}).
$$


Let $D$ be bounded and every point of $\partial D$ be regular for $D^C$.
Markov process corresponding to the Dirichlet problem is a killed Brownian
motion. We denote the eigenvalues and the eigenfunctions of
$\Delta$ by $\{\lambda_n, \phi_n\}_{n=1}^\infty$, where $\phi_n\in C^{\infty}(D)$. The
corresponding heat kernel is given by
$$
p_D(t,x,y)=\sum_{n=1}^{\infty}e^{-\lambda_n t}\phi_n(x)\phi_n(y).
$$
The series
converges absolutely and  uniformly on $[t_0,\infty)\times D\times D$ for all
$t_0>0$. In this case, the semigroup given by
\begin{equation}\label{heat-kernel}
T_D(t)f(x)=E_x[f(X_t)I(
t<\tau_D(X))]=\int_Dp_D(t,x,y)f(y)dy=\sum_{n=1}^{\infty}e^{-\lambda_n
t}\phi_n(x)\bar{f}(n)
\end{equation}
solves the Heat equation in $D$ with Dirichlet boundary conditions:
\begin{eqnarray}
\frac{\partial u(t,x)}{\partial t}&=& \Delta u(t,x),\ \ x\in D, \ t>0, \nonumber\\
u(t,x)&=&0,\ \ x\in \partial D,\nonumber\\
u(0,x)&=&f(x), \ \ x\in D.\nonumber
\end{eqnarray}

\begin{remark}\label{eigen-regularity}
The eigenfunctions belong to $L^\infty(D)\cap C^\infty(D)$,
 by \cite[Corollary 8.11, Theorems 8.15 and 8.24]{gilbarg-trudinger}. If $D$ satisfies a
 {\em uniform exterior cone condition} all the eigenfunctions belong to $C^\alpha(\bar D)$
  by \cite[Theorem 8.29]{gilbarg-trudinger}. If   $\partial D \in C^{1,\alpha}$, then all the eigenfunctions belong to $C^{1,\alpha}(\bar D)$ by \cite[Corollary 8.36]{gilbarg-trudinger}.
  If  $\partial D\in C^\infty,$
then each eigenfunction of $\Delta$ is in $C^\infty(\bar{D})$ by \cite[Theorem 8.13]{gilbarg-trudinger}.
\end{remark}





Inverse stable subordinators arise in \cite{limitCTRW,Zsolution, MNX} as scaling limits of continuous time random walks.
   Let $D(t)$ be a stable subordinator of index $0<\beta <1$. The hitting time of $D$  is defined as
   $E^\beta(t)=\inf\{x:D(x)> t\}$, so that
 $\{E^\beta(t)\leq x\}=\{D(x)\geq t\}$. We will call $E^\beta$ the process inverse to stable subordinator  $D$ in this paper.
  Writing $g_\beta(u)$ for the density of $D(1)$, it follows that $D(t)$ has density $t^{-1/\beta}g_\beta(t^{-1/\beta}u)$ for any $t>0$.  Using the inverse relation $P(E(t)\leq x)=P(D(x)\geq t)$ and taking derivatives, it follows that $E(t)$ has density
\begin{equation}\label{Etdens}
f_t(x)=t\beta^{-1}x^{-1-1/\beta}g_\beta(tx^{-1/\beta}) ,
\end{equation}
whose $t\mapsto s$ Laplace transform $s^{\beta-1}e^{-xs^\beta}$ can also be derived from the equation
\[f_t(x)=\frac d{dx}P(D(x)\geq t)=\frac d{dx}\int_t^\infty x^{-1/\beta}g_\beta(x^{-1/\beta}u)\,du\]
by taking Laplace transforms on both sides.

\section{Cauchy problems}

We prove Theorem \ref{pde-conn-n} in this section. First we need the following Lemmas.

\begin{lemma}\label{compose-et}
Let $E^{\beta_1}, E^{\beta_2} $ be two independent processes that are inverses to stable subordinators of index $0<\beta_1, \beta_2<1$.
Then $E^{\beta_1}( E^{\beta_2}(t))$ is  inverse to a stable subordinator of index $\beta_1\beta_2$.
\end{lemma}
\begin{proof}
Since composition of stable subordinators gives another stable subordinator, see Bochner \cite{bochner}, the result follows as $E^\beta$ is the inverse of a stable subordinator.
\end{proof}
\begin{lemma}\label{1-d-equivalence}
For fixed $t\geq 0$, k-iterated Brownian motion
$$
|I_{k}(t)|=|B_1(|B_2(|\cdots (|B_{k}(t)|)\cdots|)|)|
$$
and $E^{1/2^k}(t)$ have the same one-dimensional distributions.
\end{lemma}

\begin{proof}
 Since $E^{1/2}$ and $|B(t)|$ have the same $1$-dimensional distributions, see, for example,  proof of Theorem 3.1 in \cite{bmn-07}, the proof  follows from Lemma \ref{compose-et} by induction on $k$. Hence we have by composing $k$ independent $E^{1/2}$s
 $$
 E_1^{1/2}( E_2^{1/2}(\cdots(E_k^{1/2}(t))))=E^{1/2^k}(t)\stackrel{(d)}{=}|I_k(t)|.
 $$
\end{proof}

\begin{corollary}
$$
I_{k+1}(t)=B_1(|B_2(|B_3(|\cdots (B_{k+1}(t))\cdots|)|)|)\stackrel{(d)}{=}B_1(E^{1/2^k}_t)
$$

\end{corollary}

\begin{proof}[Proof of Theorem \ref{pde-conn-n}]
The proof is an adaptation of  the proof of Theorem 3.1 in Baeumer et. al. \cite{bmn-07}. Using Lemma \ref{1-d-equivalence} we get that $$u(t,x)=\E_x(f(X(|I_{k}(t)|)))=\E(f(X(E^{1/2^k}_t)))$$ is a solution to both the Higher order Cauchy problem \eqref{frac-derivative-m} and the  fractional Cauchy problem \eqref{one-third-m} for $m=2^k$. By a simple conditioning argument we also get that $u(t,x)=\E_x(f(Z(J_{k}(t))))$.
\end{proof}

\section{Cauchy problems in Bounded domains}


 The  inverse stable subordinators with $\beta=1/2^k$ are  related to Brownian subordinators by Lemma \ref{1-d-equivalence}, this is well-known for the case $k=1$, see, for example,  \cite{bmn-07}.  Since Brownian subordinators are related to higher-order Cauchy problems by Theorem \ref{pde-conn-n}, this relationship can also be used to connect those higher-order Cauchy problems in bounded domains to their time-fractional analogues.  In this section, we establish those connections for Cauchy problems on bounded domains in $\rd$.  We extend this to establish an equivalence between a killed Markov process subordinated to an inverse stable subordinator with $\beta=1/2^k$, and the same process subject to a Brownian subordinator in section 6.  Finally, we identify the boundary conditions that make the two formulations identical.  This solves an open problem in \cite{bmn-07}. This problem was solved in \cite[Theorem 4.1]{MNV} for $k=1$.


 \begin{lemma}\label{greens-2}
  Let $D$ be a bounded domain with $\partial D\in C^{1,\alpha},  0< \alpha <1.$ Let $\{\phi_n, \lambda_n, n\geq 1\}$ be the set of eigenvalues $\lambda_n$ and corresponding eigenfunctions $\phi_n$ of the
Laplacian $\Delta$.
 Let $f\in C^j(\bar D)$ for $j=1,\cdots 2k $ and all the partial derivatives of $f$ of order up to $2k-1$ vanish on the boundary (A simpler condition is to assume $f\in C^{2k}_c(D)$). Then
 \begin{equation}
 \int_D \phi_n(x)\Delta^jf(x)dx=(-\lambda_n)^j\int_D \phi_n(x)f(x)dx=(-\lambda_n)^j\bar f (n), \ j=1,\cdots k
 \end{equation}

 \end{lemma}

 \begin{proof}
 We use Green's second identity and induction in  $j$
 $$
\int_D [\phi_n\Delta^j f- \Delta ^{j-1}f\Delta \phi_n]dx=\int_{\partial
D}\left[\phi_n\frac{\partial
\Delta ^{j-1}f}{\partial \theta}-\Delta^{j-1}f\frac{\partial \phi_n}{\partial \theta}\right]ds,
$$
 where we use the fact that $\Delta^j f|_{\partial D}=0=\phi_n|_{\partial D}$, $f\in
C^j(\bar{D})$ for $j=1,\cdots , 2k$, and $\phi_n \in C^{1,\gamma}(\bar{D})$ by Remark \ref{eigen-regularity}. Hence,    by induction, the $\phi_n$-transform of $\Delta^j u$ is $(-\lambda_n)^j\bar f (n)$,
as $\phi_n$ is the eigenfunction of the Laplacian corresponding to
 eigenvalue $\lambda_n$.
 \end{proof}
 \begin{proof}[Proof of Theorem \ref{ibm-pde-thm}]

 Suppose $u$ is a solution to  Equation (\ref{ibm-pde-bounded-02}). Taking the $\phi_n$- transform of (\ref{ibm-pde-bounded-02}) and using Lemma \ref{greens-2}, we obtain

\begin{equation}\label{sine-trans-1}
\frac{\partial}{\partial t}\bar{u}(t,n) =\sum_{j=1}^{m-1} \frac{t^{j/m-1}}{\Gamma(j/m)} (-\lambda_n)^j \bar{f}(n)
+ (-\lambda_n)^m \bar{u}(t,n).
\end{equation}

Note that the time derivative commutes with the $\phi_n$-transform, as
$$
\left|\frac{\partial}{\partial t}u(t,x)\right|\leq  g(x)t^{ 1/m-1}, g\in L^\infty(D), \ t>0.
$$

Taking Laplace transforms on both sides and  using the well-known Laplace transform formula
\begin{equation}\label{power-laplace}
\int_0^\infty \frac{t^{-\beta}}{\Gamma(1-\beta)}e^{-st}dt=s^{\beta-1}
\end{equation}
for $\beta<1,$ gives us
\begin{equation}\label{laplace-sine-trans-1}
s\hat{u}(s,n)-\bar u(0,n) =\sum_{j=1}^{m-1} s^{-j/m} (-\lambda_n)^j \bar{f}(n)
+ (-\lambda_n)^m \hat{u}(s,n).
\end{equation}
Since $u$ is uniformly continuous on $C(
[0,\epsilon]\times \bar D)$, it is also uniformly bounded on $
[0,\epsilon]\times \bar D$. So, we have $\lim_{t\to 0}\int_D u(t,x)\phi_n(x)dx=\bar
f(n)$.  Hence, $\bar{u}(0,n)=\bar f(n)$.
\noindent
By collecting the like terms, we obtain
\begin{equation}\label{direct-ibm}
\hat{u}(s,n)=\frac{\bar{f}(n)\left( 1+\sum_{j=1}^{m-1} s^{-j/m} (-\lambda_n)^j \right)}{s-(-\lambda_n)^m}.
\end{equation}
Using the simple equality $(a-b)(a^{m-1} +a^{m-2}b+\cdot +ab^{m-1}+b^{m-1})=a^m-b^m$ for any $m=2,3,\cdots$,
for fixed $n$ and  for large $s,$ we get
\begin{eqnarray}
\hat{u}(s,n)&=&\frac{s^{1/m-1}\bar{f}(n)\left( s^{1-1/m}+\sum_{j=1}^{m-1} s^{-j/m+1-1/m} (-\lambda_n)^j \right)}{s-(-\lambda_n)^m}\label{LFT}\\
&=&\frac{s^{1/m-1}\bar{f}(n)\left( s^{1-1/m}+\sum_{j=1}^{m-1} s^{-j/m+1-1/m} (-\lambda_n)^j \right)}{(s^{1/m}-(-\lambda_n))\left( s^{1-1/m}+\sum_{j=1}^{m-1} s^{-j/m+1-1/m} (-\lambda_n)^j \right)} \nonumber\\
&=& \frac{s^{1/m-1}\bar{f}(n)}{s^{1/m}+\lambda_n}.\label{frac-phi-laplace}
\end{eqnarray}



It follows from the proof of Theorem 3.1 and Corollary
3.2 in \cite{MNV} that the inverse $\phi_n$-Laplace transform of \eqref{frac-phi-laplace} is given by
\begin{eqnarray}
u(t,x)&=&\sum_{n=1}^{\infty}\bar{f}(n) E_{1/m}(-\lambda_n t^{1/m}) \phi_{n}(x)\label{ibm-series}\\
&=&\E_x[f(X(E^{1/m}(t))) I(\tau_D(X)>E^{1/m}(t))]\nonumber\\
&=&\E_x[f(X(E^{1/m}(t))) I(\tau_D(X(E^{1/m}))>t)],\nonumber
\end{eqnarray}
where $E^{1/m}(t)$ is the process inverse to  stable subordinator of index $1/m$.

 For any fixed $n\geq 1$, the two
formulae \eqref{direct-ibm} and \eqref{frac-phi-laplace} are well-defined and
equal for all sufficiently large $s$. Since the inverse Laplace transform of $\hat u(s,n)$ in \eqref{frac-phi-laplace} is
\begin{equation}\label{mmm3}
\bar{u}(t,n)=\bar{f}(n)E_{1/m}(-\lambda_nt^{1/m}),
\end{equation}
 see, for example, \cite{MG-ML}, we can see easily
  that $\bar{u}(t,n)$ is continuous in $t>0$ for any $n\geq 1$.
Hence, the uniqueness  theorem for Laplace transforms \cite[Theorem
1.7.3]{ABHN} shows that, for each $n\geq 1$,  $\bar{u}(t,n)$  is the
unique continuous function whose Laplace transform is given by
 \eqref{LFT}.  Since $x\mapsto u(t,x)$ is an
element of $L^2(D)$ for every $t>0$, and  two elements of
$L^2(D)$ with the same $\phi_n$-transform are equal $dx$-almost
everywhere, we have  \eqref{ibm-series} is the unique element of $L^2(D)$
 and  \eqref{mmm3} is its $\phi_n$-transform.

Now, we show that the solution $u$ defined by \eqref{ibm-series} satisfies all the properties in \eqref{ibm-pde-bounded-02}.
From the Proof of Theorem 3.1 in \cite{MNV}, we can  get that the
solution $u(t, x)$  defined by the series \eqref{ibm-series} belongs to $L^2(D)$, converges absolutely and uniformly for $t\geq t_0>0$ for some $t_0$.

Next we  show that $\Delta ^l u
\in C(D_\infty)$  for $l=1,\cdots , m$.  To do this, we need only to show the absolute and uniform convergence of the
series defining $\Delta^l u$ for $l=1,\cdots , m$.  To apply $\Delta^l$ term-by-term to
\eqref{ibm-series}, we have to show that the series
$$
\sum_{n=1}^{\infty}\bar{f}(n)\phi_{n}(x)(-\lambda_n)^lE_{1/2}(-\lambda_n
t^{1/m})
$$
is  absolutely and uniformly convergent for $t>t_0
>0.$

Note that by Lemma \ref{greens-2} the $\phi_n$-transform of $\Delta^j f$ is given by
$$
\int_D\phi_n(x)\Delta^j f(x)dx =(-\lambda_n)^j\bar f(n)
$$
and  using
\cite[equation (13)]{krageloh},
\begin{equation}
\label{efie}
0\leq E_\beta(-\lambda_n t^\beta)\leq
c/(1+\lambda_n t^\beta),
\end{equation}
 we get
\begin{eqnarray}
\sum_{n=1}^{\infty}|\bar{f}(n)||\phi_{n}(x)|(\lambda_n)^lE_{1/m}(-\lambda_n
t^{1/m})&\leq &
\sum_{n=1}^{\infty}|\bar{f}(n)||\phi_{n}(x)|\lambda_n^l\frac{c}{1+\lambda_n
t^\frac{1}{m}}\nonumber\\
&\leq & c
t_0^{-\frac{1}{m}}\sum_{n=1}^{\infty}|\bar{f}(n)||\phi_{n}(x)|\lambda_n^{l-1}<\infty,
\end{eqnarray}
where the last inequality follows from the absolute and uniform
convergence of the eigenfunction expansion of $\Delta^{m-1} f$.

A  similar argument  using
\cite[Equation (17)]{krageloh}
\begin{equation}\label{time-derivative}
\left| \frac{ d E_\beta(-\lambda_n
t^\beta)}{dt}\right|\leq c\frac{\lambda_n t^{\beta-1}}{1+\lambda_n t^\beta}\leq c\lambda_nt^{\beta -1},
\end{equation}
and the fact that  the eigenfunction expansion of $\Delta^{m-1} f$ converges absolutely and uniformly allows us to differentiate the series \eqref{ibm-series} term by term with respect to $t$.

We next show that   $\Delta^l u \in C^1(\bar D)$ for $l=1,\cdots , m-1$: this follows from the bounds
in \cite[Theorem 8.33]{gilbarg-trudinger}
 and the absolute and uniform convergence of the series defining $\Delta^{m-1}f$.

 \begin{eqnarray}
 |\phi_n|_{1,\alpha; D}&\leq & C(1+\lambda_n)\sup_{D}|\phi_n(x)|,\label{uniform-eigenvalue-bounds}
 \end{eqnarray}
where $C=C(d,\partial D)$ is a finite constant. Here $$|u|_{k,\alpha; D}=\sup_{|\gamma|=k}[D^\gamma u]_{\alpha ,D}+ \sum_{j=0}^{k} \sup_{|\gamma|=j}\sup_{D}|D^\gamma u|,\ \ k=0,1,2,\cdots$$ and
$$[D^\gamma u]_{\alpha ,D}=\sup_{x,y\in D, x\neq y}\frac{|D^\gamma u(x)- D^\gamma u(y)|}{|x-y|^\alpha}$$
are norms on $C^{k,\alpha}(\bar D)$.
Hence for $l=0, 1, 2,\cdots, m-1$
\begin{eqnarray}
|\Delta^lu(.,t)|_{1,\alpha; D}&\leq & C \sum_{n=1}^\infty |\bar f(n)|E_\beta(-\lambda_n
t^\beta)(\lambda_n)^l(1+\lambda_n)\sup_{D}|\phi_n(x)| \nonumber\\
&\leq &C \sum_{n=1}^\infty |\bar f(n)(\lambda_n)^l|\frac{1+\lambda_n}{1+\lambda_n t^\beta}\sup_{D}|\phi_n(x)| \nonumber\\
&\leq & Ct^{-\beta} \sum_{n=1}^{\infty}\sup_{D}|\phi_n(x)| (\lambda_n)^l|\bar{f}(n)|\nonumber\\
&&+C\sum_{n=1}^{\infty}\sup_{D}|\phi_n(x)|(\lambda_n)^l|\bar{f}(n)| <\infty .\nonumber
 \end{eqnarray}

 With  this we established that $u\in \mathcal{H}_{\Delta^m}\cap C_b(D_\infty)\cap C^1(\bar D)$.


Observe next that the Laplace transform of
$$
\frac{\partial}{\partial t}E_\beta(-\lambda_n t^{1/m})-
\sum_{j=1}^{m-1} \frac{t^{j/m-1}}{\Gamma(j/m)} (-\lambda_n)^j
-(-\lambda_n)^m E_\beta(-\lambda_n t^{1/m})
$$
 is
$$
\frac{s^{1/m}}{s^{1/m}+\lambda_n}-1+\sum_{j=1}^{m-1} s^{-j/m} (-\lambda_n)^j-\frac{(-\lambda_n)^ms^{1/m-1}}{s^{1/m}+\lambda_n}=0,
$$
since the Laplace transform of $E_{1/m}(-\lambda_n
t^{1/m})$ is $\frac{s^{1/m-1}}{s^{1/m}+\lambda_n}$
 and using \eqref{power-laplace}. Hence, by the uniqueness of Laplace transforms, we get
$$
\frac{\partial}{\partial t}E_\beta(-\lambda_n t^{1/m})-
\sum_{j=1}^{m-1} \frac{t^{j/m-1}}{\Gamma(j/m)} (-\lambda_n)^j
-(-\lambda_n)^m E_\beta(-\lambda_n t^{1/m})=0.
$$

\noindent
  Now applying the
  time derivative and  $\Delta^m$ to the series in (\ref{ibm-series}) term by term gives
\begin{eqnarray}
&&\frac{\partial}{\partial t}u(t,x) -\sum_{j=1}^{m-1} \frac{t^{j/m-1}}{\Gamma(j/m)} \Delta^j f(x)
- \Delta^m u(t,x)\nonumber\\
&= &\sum_{n=1}^\infty \bar
f(n)\left[\phi_n(x)\frac{\partial}{\partial
t}E_{\frac{1}{m}}(-\lambda_n t^{\frac{1}{m}})- \sum_{j=1}^{m-1} \frac{t^{j/m-1}}{\Gamma(j/m)} \Delta^j \phi_n(x)-E_{\frac{1}{m}}(-\lambda_n
t^{\frac{1}{m}})\Delta^m\phi_n(x)\right]\nonumber\\
&=& \sum_{n=1}^\infty \bar
f(n)\phi_n(x) \left[\frac{\partial}{\partial
t}E_{\frac{1}{m}}(-\lambda_n t^{\frac{1}{m}})+ - \sum_{j=1}^{m-1} \frac{t^{j/m-1}}{\Gamma(j/m)} (-\lambda_n)^j-(-\lambda_n)^mE_{\frac{1}{m}}(-\lambda_n
t^{\frac{1}{m}})\right]\nonumber\\
&=& 0\nonumber
\end{eqnarray}
which shows that the PDE in (\ref{ibm-pde-bounded-02}) is satisfied.
  Thus, we  conclude that $u$ defined by
(\ref{ibm-series}) is a classical solution to
(\ref{ibm-pde-bounded-02}). This completes the first part of the proof.

We next deal with the case $m=2^k$. Observe that $E^{1/2^k}(t)$, process inverse to stable subordinator of index $1/2^k$ and $|I_k(t)|$ have the  same density, as they have the same one dimensional distribution by Lemma \ref{1-d-equivalence}. Let $p(t,l)$ denote the common density of $|I_k(t)|$ and $E^{1/2^k}(t)$.
\begin{eqnarray}
u(t,x)&=&\sum_{n=1}^{\infty}\bar{f}(n) E_{1/2^k}(-\lambda_n t^{1/2^k}) \phi_{n}(x)\label{ibm-series1}\\
&=&\int_{0}^{\infty}\left[\sum_{n=1}^{\infty}\bar{f}(n) e^{-\lambda_n l} \phi_{n}(x)\right]
 2^ktg_{1/2^k}({tl^{-2^k}})l^{-(1+2^k)}dl\nonumber\\
&=&\int_{0}^{\infty}\left[\sum_{n=1}^{\infty}\bar{f}(n) e^{-\lambda_n l} \phi_{n}(x)\right] p(t,l)dl\label{e-y-equivalence}\\
&=&\int_{0}^{\infty}T_D(l)f(x) p(t,l)dl\nonumber\\
&=& \E_x[f(X(|I_k(t)|)) I(\tau_D(X)>|I_k(t)|)].\label{condition-Y}
\end{eqnarray}
Note that Equation \eqref{condition-Y} follows
by a conditioning argument.

Finally we prove uniqueness. Let $u_1, u_2$
be two solutions of \eqref{ibm-pde-bounded-02}
with initial data $u(0,x)=f(x)$ and Dirichlet boundary condition $u(t,x)=0$ for $x\in\partial D$. Then $U=u_1-u_2$
is a solution of \eqref{ibm-pde-bounded-02} with zero
initial data and zero boundary value.  Taking
$\phi_n$-transform on both sides of
\eqref{ibm-pde-bounded-02} we get
$$
\frac{\partial}{\partial t}\bar{U}(t,n) =
 (-\lambda_n)^m
\bar{U}(t,n), \ \ \bar{U}(0,n)=0,
$$
and then $\bar U(t,n)=0$ for all $t>0$ and all $n\geq
1$. This implies that $U(t,x)=0$ in the sense of $L^2$ functions,
since $\{\phi_n: \ n\geq 1\}$ forms a complete
orthonormal basis for $L^2(D)$. Hence, $U(t,x)=0$ for all $t>0$
and almost all $x\in D$. Since $U$ is a continuous function on
$D$, we have $U(t,x)=0$ for all $(t,x)\in [0,\infty)\times D,$
thereby proving uniqueness.
 \end{proof}

 \begin{corollary}
Let $f\in C^{2k}_c(D)$ be a  $2k$-times continuously differentiable
function of compact support in D. If $k>m-1+3d/4$, then the Equation
(\ref{ibm-pde-bounded-02}) has a strong solution. In particular, if
$f\in C^{\infty}_c(D)$, then the classical solution of Equation
(\ref{ibm-pde-bounded-02}) is in $C^\infty(D).$

\end{corollary}

\begin{proof} By Example 2.1.8 of \cite{davies},  $|\phi_n(x)|\leq (\lambda_n)^{d/4}$.
Also,  from Corollary 6.2.2 of \cite{davies-d}, we have $\lambda_n\sim  n^{2/d}$.

By  Lemma \ref{greens-2} we get

\begin{equation}\label{phi-trans-laplacian}
\overline{\Delta^{k} f}(n)=(-\lambda_n)^{k} \bar{f}(n).
\end{equation}
Using  Cauchy-Schwartz inequality and the fact $f\in C^{2k}_c(D),$ we get
   $$\overline{\Delta^{k} f}(n)\leq \left[\int_D (\Delta^{k} f(x))^2dx\right]^{1/2}\left[\int_D (\phi_n(x))^2dx\right]^{1/2}=\left[\int_D (f^{(2k)}(x))^2dx\right]^{1/2}= c_k,$$
 where $c_k$ is  a constant independent of $n$.

This and Equation \eqref{phi-trans-laplacian} give $|\bar{f}(n)|\leq c_k(\lambda_n)^{-k}$.

Since
$$
\Delta^{m-1} f(x)=\sum_{n=1}^\infty (-\lambda_n)^{m-1}\bar f (n) \phi_n(x),
$$
to get the
absolute and uniform convergence of the series defining $\Delta^{m-1} f$, we
consider
   \begin{eqnarray*}
   \sum_{n=1}^{\infty}(\lambda_n)^{m-1}|\phi_n(x)|| \bar{f}(n)|&\leq& \sum_{n=1}^{\infty} (\lambda_n)^{d/4+m-1}c_k(\lambda_n)^{-(k)}\\
   &\leq& c_n\sum_{n=1}^{\infty} (n^{2/d})^{d/4+m-1-k}=c_k\sum_{n=1}^{\infty} n^{1/2+2(m-1)/d-2k/d}
   \end{eqnarray*}
 which is finite if $-2(m-1)/d-1/2+2k/d>1$, i.e. $k>m-1+3d/4$.
 \end{proof}
 \section{Other higher order Cauchy Problems}


 The $d$-dimensional symmetric $\alpha$-stable process $X({t})$ with
$\alpha\in (0,2]$ is the process with stationary independent
increments whose transition density
$$
p_{t}^{\alpha}(x,y)=p^{\alpha}(t,x-y), \ \ \ \ \ (t,x,y)\in
(0,\infty)\times \RR{R}^{d}\times \RR{R}^{d},
$$
is characterized by
$$
\int_{\RR{R}^{d}}
e^{iy.\xi}p^{\alpha}(t,y)dy=\exp(-t|\xi|^{\alpha}), \ \ \ \ \ t>0,
\xi \in \RR{R}^{d}.
$$
The process has right continuous paths, it is rotation and
translation invariant. For $\alpha =2$, this is Brownian motion running twice  the speed of standard Brownian motion.

Since the Laplace transform method does not apply by the appearance of $t^{-1}$ in the PDE \eqref{PDE-CONNECT2-bounded}, we use a direct method to prove Theorem \ref{cauchy-bounded}.

\begin{proof}[Proof of Theorem \ref{cauchy-bounded}]
By a simple conditioning argument and using the series representation of the killed semigroup $T_D(t)$ in \eqref{heat-kernel}, we can express $u(t,x)$ as
\begin{eqnarray}
u(t,x)&=&2\int_0^\infty \bigg(\sum_{n=1}^\infty e^{-\lambda_ns}\bar f(n)\phi_n(x)\bigg)p^1(t,s)ds\nonumber\\
&=&2\int_0^\infty \bigg(\sum_{n=1}^\infty e^{-\lambda_ns}\bar f(n)\phi_n(x)\bigg)\frac{t}{\pi(t^2+s^2)}ds\label{cauchy-series}.
\end{eqnarray}

We use Fubini-Tonelli theorem,  the simple inequality
\begin{equation}\label{cauchy-bound}
\int_0^\infty  e^{-\lambda_ns}\frac{t}{\pi(t^2+s^2)}ds\leq \frac{1}{\pi t\lambda_n},
\end{equation}
and the fact that the series defining $\Delta f$ is absolutely and uniformly convergent to show that
 $$u(t,x)=2 \sum_{n=1}^\infty \bar f(n)\phi_n(x)\int_0^\infty e^{-\lambda_ns}\frac{t}{\pi(t^2+s^2)}ds.$$
Using this, and \eqref{cauchy-bound}
we can show that we can apply $\Delta^2$ to the series \eqref{cauchy-series} term by term. Hence it follows that  $u, \Delta u, \Delta^2 u\in C(D_{\infty})$.

We next show that each term in the series \eqref{cauchy-series} satisfy the PDE \eqref{cauchy-bounded}.
We use the fact that $p^{1}(t,s)$ satisfy
$$
(\frac{\partial^{2}}{\partial s^{2}}+\frac{\partial^{2}}{\partial
t^{2}})p^{1}(t,s)=0,
$$
dominated convergence theorem, and integration by parts twice to get
\begin{eqnarray}
\frac{\partial^{2}}{\partial
t^{2}}\bigg(\bar f(n)\phi_n(x) \int_0^\infty e^{-\lambda_ns}p^{1}(t,s)ds\bigg)&=&\bar f(n)\phi_n(x) \int_0^\infty e^{-\lambda_ns}\frac{\partial^{2}}{\partial
t^{2}}p^{1}(t,s)ds\nonumber\\
 &=&- \bar f(n)\phi_n(x) \int_0^\infty e^{-\lambda_ns}\frac{\partial^{2}}{\partial
s^{2}}p^{1}(t,s)ds \nonumber\\
& =& \frac{1}{\pi t}\lambda_n\phi_n(x)\bar f (n)\nonumber\\
&+&\lambda_n^2\phi_n(x)\bar f (n)\int_0^\infty e^{-\lambda_ns}p^{1}(t,s)ds \nonumber\\
&=&-\frac{1}{\pi t}\Delta \phi_n(x)\bar f (n)\nonumber\\
&+&-\Delta^2\phi_n(x)\bar f (n)\int_0^\infty e^{-\lambda_ns}p^{1}(t,s)ds. \nonumber
\end{eqnarray}
From this we get that the time derivative can be applied term by term to the series \eqref{cauchy-series} since the series for $\Delta f$ and $\Delta^2 u$ converge absolutely and uniformly.
Hence applying the time derivative and the $\Delta ^2$ term by term to the series \eqref{cauchy-series} we obtain
\begin{eqnarray}
&&\frac{\partial ^{2}}{\partial t^{2}}u(t,x)\ +\frac{2\Delta f(x)}{\pi t}\ + \ \Delta^{2}u(t,x)\nonumber\\
&=& \sum_{n=1}^\infty \bar f(n)\bigg(\phi_n(x) \int_0^\infty e^{-\lambda_ns}\frac{\partial^{2}}{\partial
t^{2}}p^{1}(t,s)ds + \frac{1}{\pi t}\Delta \phi_n(x)\nonumber\\
&&\ \ \ \ \ \ \ \ \ \ \ \ \ +\Delta^2\phi_n(x)\int_0^\infty e^{-\lambda_ns}p^{1}(t,s)ds \bigg)=0.
\end{eqnarray}


\end{proof}


For a rational $\alpha \neq 1$  the PDE is more complicated since  kernels
of symmetric $\alpha$-stable processes satisfy a higher order PDE.

\begin{theorem}[Nane \cite{nane-h}]\label{alphapde}
Let $\alpha \in (0,2)$ be rational $\alpha=l/m$, where $l$ and $m$
are relatively prime. Let $T(s)f(x)=\E_x[f(X(s))]$ be the
semigroup of Brownian motion $X(t) $ and let
$\Delta$ be its generator. Let $f$ be a bounded measurable
function in the domain of $\Delta$, with $D^{\gamma}f$  bounded and
H\"{o}lder continuous for all multi index $\gamma$ such that $|\gamma |=2l$. Then
$u(t,x)=\E_x[f(X(|Y(t)|)]$
is a solution of
\begin{eqnarray}
(-1)^{l+1}\frac{\partial ^{2m}}{\partial t^{2m}}u(t,x) & =  &
-2\sum_{i=1}^{l} \left(\frac{\partial^{2l-2i}}{\partial
s^{2l-2i}}p^{\alpha}(t,s)\big|_{s=0}\right) \Delta^{2i-1}f(x)\nonumber\\
 & &\ -\ \Delta^{2l} u(t,x), \ \ \ \ t>0,\ \   x\in \rd \nonumber\\
u(0,x) &  =  & \ f(x), \ \ \ \ \ x \in \RR{R}^{d}.\nonumber
\end{eqnarray}
\end{theorem}
Let
\begin{eqnarray*}
\mathcal{K}_{\Delta ^{2l}} &\equiv &\bigg\{u:D_\infty\to \rr:\ \ \frac{\partial^2}{\partial t^2}u, \Delta^j u, \in C(D_\infty) \   \mathrm{for}\  j=1,
 \cdots, 2l,\\
 && \ \ \ \ \Delta^j u, \in C(\bar D)\   \mathrm{for}\ j=1,\cdots, 2l-1\bigg\}
\end{eqnarray*}

This theorem takes the following form in bounded domains
\begin{theorem}\label{alphapde-bounded}
Let $\alpha \in (0,2)$ be rational $\alpha=l/m$, where $l$ and $m$
are relatively prime.
Let $D$ be a bounded domain with $\partial D\in C^{1,\gamma},  0< \gamma <1.$
Let
 $\{X(t)\}$ be a Brownian motion, and let $\{Y(t)\}$ be symmetric $\alpha$-stable process independent of $\{X(t)\}$. Let $f\in D(\Delta)\cap C^{4l-2}(D)$ for which the eigenfunction expansion of $\Delta^{2l-1} f$  with respect to the complete orthonormal basis $\{\phi_n:\ n\geq 1\}$ converges absolutely and uniformly.
Then
$u(t,x)=\E_x[f(X(|Y(t)|)I(\tau_D(X)>|Y(t)|)]$
is a classical solution of
\begin{eqnarray}
u&\in &\mathcal{K}_{\Delta ^{2l}}\cap C_b(D_\infty)\cap C^1(\bar D)\nonumber\\
(-1)^{l+1}\frac{\partial ^{2m}}{\partial t^{2m}}u(t,x)\ & =  &
-2\sum_{i=1}^{l} \left(\frac{\partial^{2l-2i}}{\partial
s^{2l-2i}}p^{\alpha}(t,s)\big|_{s=0}\right) \Delta^{2i-1}f(x)\
  \label{alpha-higher-pde}\\
 & &\ -\ \Delta^{2l} u(t,x), \ \ \ \ t>0,\ \   x\in D \nonumber\\
u(0,x) &  =  & \ f(x), \ \ \ \ \ x \in D.\nonumber\\
u(t,x)& = & \Delta^ju(t,x) = 0,\ \ x\in \partial D, \ t\geq 0,\ j=1,\cdots, 2l-1.\nonumber
\end{eqnarray}
\end{theorem}

\begin{proof}
Let $T_D(t)$ be the killed semigroup of Brownian motion in $D$.
We use the representation
\begin{eqnarray}
u(t,x)&=&\E_x[f(X(|Y(t)|)I(\tau_D(X)>|Y(t)|)]=2\int_{0}^{\infty}p^{\alpha}(t,s)
T_D({s})f(x)ds\nonumber\\
&=&2\int_0^\infty \bigg(\sum_{n=1}^\infty e^{-\lambda_ns}\bar f(n)\phi_n(x)\bigg)p^\alpha(t,s)ds\label{alpha-series}
\end{eqnarray}
and the fact that the transition density $p^{\alpha}(t,s)$ of the process $Y(t)$
satisfies
$$
(\frac{\partial^{2}}{\partial
s^{2}})^{l}+(-1)^{l+1}\frac{\partial^{2m}}{\partial
t^{2m}})p^{\alpha}(t,s)=0,\ \ \ \ \ (t,x)\in (0,\infty)\times
\rr,
$$
 from Lemma 3.2 in \cite{nane-h}.

 We also use the  well-known fact  that for $\alpha\in(0,2)$
$$
p^\alpha(t,x)=t^{-d/\alpha}p(1,t^{-1/\alpha}x)\leq t^{-d/\alpha}p(1,0)=t^{-d/\alpha}M_{d,\alpha}, t>0, x\in \rd,
$$
where
$$
M_{d,\alpha}=\frac{1}{(2\pi)^d}\int_\rd e^{-|x|^\alpha}dx.
$$
Hence
$$
\int_0^\infty e^{-\lambda_n s}p^{\alpha}(t,s)ds\leq \frac{M_{d,\alpha}}{\lambda_nt^{d/\alpha}}.
$$
This allows us to deduce the fact that $\Delta^j u\in C(D_\infty) $ for $j=0, 1,\cdots, 2l$ and also  that $\Delta^j u\in C(\bar D) $ for $j=0, 1,\cdots, 2l-1$.

To get that $\frac{\partial ^{2m}}{\partial t^{2m}}u(t,x)\in C(D_\infty)$ we use the fact  that the series defining $\Delta^{2l}u$ converges absolutely and uniformly as well as the series defining $\Delta ^{2l-1}f$ converges absolutely and uniformly and the fact that the terms in the series defining $u(t,x)$ in \eqref{alpha-series} satisfy the PDE \eqref{alpha-higher-pde}.

Hence we can interchange the sum and powers of the Laplacian in  the series \eqref{alpha-series} term by term to show that PDE in \eqref{alpha-higher-pde} is satisfied.
\end{proof}

\newpage
 \section{Extensions and Discussion}

 A uniformly elliptic operator of divergence form is defined on
$C^2$ functions by
\begin{equation}\label{unif-elliptic-op}
Lu=\sum_{i,j=1}^{d}\frac{\partial \left(a_{ij}(x)(\partial
u/\partial x_i)\right)}{\partial x_j}
\end{equation}
with $a_{ij}(x)=a_{ji}(x)$ and, for some $\lambda>0,$
\begin{equation}\label{elliptic-bounds}
\lambda \sum_{i=1}^ny_i^2\leq \sum_{i,j=1}^na_{ij}(x)y_iy_j\leq
\lambda^{-1} \sum_{i=1}^ny_i^2,\ \ \forall y \in \rd.
\end{equation}

The operator $L$ acts on the Hilbert space $L^2(D)$. We define the initial domain $C_0^\infty(\bar D)$ of the operator as follows.  We say that $f$ is in $C_0^\infty(\bar D),$ if  $f\in C^\infty(\bar D)$ and $f(x)=0$ for all $x\in \partial D$. This condition incorporates the notion of  Dirichlet boundary conditions.

From \cite[Corollary 6.1]{davies-d}, we have that the associated
quadratic form
$$
Q(f,g)=\int_D\sum_{i,j=1}^d a_{ij}\frac{\partial f}{\partial
x_i}\frac{\partial g}{\partial x_j}dx
$$
 is closable on the domain $C_0^\infty(\bar D)$ and the domain of the closure is independent of the particular coefficients $(a_{ij})$ chosen. In particular, for $f,g\in C_0^\infty(\bar D)$ by integration by parts
 $$
 \int_D g(x)Lf(x)dx=Q(f,g)=\int_D f(x)Lg(x)dx,
 $$
which shows that $L$ is symmetric.

From now on, we will use the symbol $L_D$ if we particularly want to
emphasize the choice of Dirichlet boundary conditions, to refer to
the self-adjoint operator associated with the closure of the
quadratic form above by the use of \cite[Theorem 4.4.5]{davies-d}.
Thus, $L_D$ is the Friedrichs extension of the operator defined
initially on $C_0^\infty(\bar D)$.

If the coefficients $a_{ij}(x)$ are smooth ($a_{ij}(x) \in C^1(D)$), then  $L_Du$ takes the form
$$
L_Du=\sum_{i,j=1}^{d}a_{ij}(x)\frac{ \partial^2
u}{\partial x_i\partial  x_j}+\sum_{i=1}^d\left(\sum_{j=1}^{d}\frac{\partial a_{ij}(x)}{\partial x_j}\right)\frac{\partial
u}{\partial x_i}=\sum_{i,j=1}^{d}a_{ij}(x)\frac{ \partial^2
u}{\partial x_i\partial  x_j}+\sum_{i=1}^d b_i(x)\frac{\partial
u}{\partial x_i}.
$$

If $X_t$ is a solution to
$$
dX_t=\sigma (X_t)dW_t+b(X_t)dt, \ \ X_0=x_0,
$$
where $\sigma$ is a $d\times d$ matrix, and $W_t$ is a Brownian
motion, then $X_t$ is associated with the operator $L_D$ with
$a=\sigma \sigma ^T$, (see Chapters 1 and 5 of
Bass \cite{bass}). Define the first exit time as
 $\tau_D(X)=\inf \{ t\geq 0:\ X_t\notin D\}$. The semigroup defined by
$T(t)f(x)=E_x[f(X_t)I(\tau_D(X))>t)]$ has generator $L_D$, which
follows by an application of the It$\mathrm{\hat{ o}}$ formula.

Let $D $ be a  bounded domain in $\RR{R}^d$. Suppose $L$ is a uniformly elliptic operator of divergence form
with Dirichlet
  boundary conditions on $D$, and that  there exists a constant $\Lambda$ such that for all $x\in D$,
\begin{equation}\label{L-uniform-bound}
\sum_{i,j=1}^d |a_{ij}(x)|\leq \Lambda.
\end{equation}
Let
  $T_D(t)$ be the corresponding semigroup. Then
 $T_D(t)$ is an ultracontractive semigroup (even  an intrinsically
 ultracontractive semigroup), see Corollary 3.2.8, Theorem 2.1.4, Theorem 4.2.4, and Note 4.6.10 in \cite{davies}.
  Every ultracontractive semigroup has a kernel for the killed semigroup on a
   bounded domain which can be represented as a series expansion of the eigenvalues
   and the eigenfunctions of $L_D$  (cf. \cite[Theorems 2.1.4 and 2.3.6]{davies} and
    \cite[Theorems 8.37 and 8.38]{gilbarg-trudinger} ): There exist eigenvalues
$0< \mu_1<\mu_2\leq \mu_3\cdots,$ such that $\mu_n\to\infty,$ as
$n\to\infty$, with the  corresponding complete orthonormal set (in $H^2_0$) of
eigenfunctions $\psi_n$ of the operator $L_D$ satisfying

\begin{equation}\label{eigen-eigen}
 L_D \psi_n(x)=-\mu_n \psi_n(x), \ x\in D:\  \\
\psi_n |_{\partial D}=0.
\end{equation}

 In this case,
$$
p_D(t,x,y)=\sum_{n=1}^{\infty}e^{-\mu_n t}\psi_n(x)\psi_n(y)
$$
is the heat kernel of the killed semigroup $T_D$. The series
converges absolutely and  uniformly on $[t_0,\infty)\times D\times D$ for all
$t_0>0$.

 Suppose $D$  satisfies a uniform exterior cone condition. Let $\{X_t\}$ be a Markov process
 in $\rd$ with generator $L_D$, and $f$
be continuous on $\bar D$. Then the semigroup
\begin{equation}\begin{split}\label{TDdef}
T_D(t)f(x)&=E_x[f(X_t)I( t<\tau_D(X))]\\
&=\int_D p_D(t,x,y)f(y)dy\\
&=\sum_{n=1}^{\infty}e^{-\mu_n t}\psi_n(x)\bar{f}(n)
\end{split}\end{equation}
 solves the Dirichlet initial-boundary value problem in $D$:
\begin{eqnarray}
\frac{\partial u(t,x)}{\partial t}&=& L_D u(t,x),\ \ x\in D, \ t>0, \nonumber\\
u(t,x)&=&0,\ \ x\in \partial D,\nonumber\\
u(0,x)&=&f(x), \ \ x\in D.\nonumber
\end{eqnarray}





\begin{remark}\label{diffusion-eigenfunctions}
The eigenfunctions belong to $L^\infty(D)\cap C^\alpha(D)$ for some $\alpha>0$,
 by \cite[Theorems 8.15 and 8.24]{gilbarg-trudinger}. If $D$ satisfies a
 {\em uniform exterior cone condition} all the eigenfunctions belong to $C^\alpha(\bar D)$
  by \cite[Theorem 8.29]{gilbarg-trudinger}. If $a_{ij}\in C^\alpha(\bar D)$ and  $\partial D \in C^{1,\alpha}$, then all the eigenfunctions belong to $C^{1,\alpha}(\bar D)$ by \cite[Corollary 8.36]{gilbarg-trudinger}.
  If $a_{ij} \in C^{\infty}(D)$ then each eigenfunction of $L$ is in $C^\infty(D)$ by \cite[Corollary 8.11]{gilbarg-trudinger}.
If $a_{ij} \in C^{\infty}(\bar{D})$ and $\partial D\in C^\infty,$
then each eigenfunction of $L$ is in $C^\infty(\bar{D})$ by \cite[Theorem 8.13]{gilbarg-trudinger}.
\end{remark}

\begin{remark}Theorems \ref{pde-conn-n}, \ref{ibm-pde-thm} and \ref{cauchy-bounded} are valid if we replace the outer
 Brownian motion with a Diffusion process.
 This can be verified by considering \eqref{uniform-eigenvalue-bounds}, \eqref{eigen-eigen}, \eqref{TDdef}, Remark \ref{diffusion-eigenfunctions},
  and Theorem 3.1 in \cite{MNV}.
\end{remark}

\begin{remark}
It might be an interesting project to consider the PDEs treated in this paper with the Neumann boundary conditions. Probably the solutions will be obtained by running reflected diffusions subordinated by $k$-iterated Brownian motions. We will treat this problem elsewhere.
\end{remark}

\textbf{Acknowledgments.} I would like to thank Mark Meerschaert for encouragement on working on the problems in this paper. I also would like to thank anonymous referee for his or her helpful comments on the comments on the historical priority on the equivalence of fractional Cauchy problems and higher order PDEs which improved the accuracy of the results leading to the present paper.

\end{document}